\documentclass{amsart}
\usepackage{amsmath}
  \usepackage{paralist}
  \usepackage{graphics} 
  \usepackage{epsfig} 
 \usepackage[colorlinks=true]{hyperref}

\hypersetup{urlcolor=blue, citecolor=red}
\def\MR#1{\href{http://www.ams.org/mathscinet-getitem?mr=#1}{MR#1}}

\newtheorem{theorem}{Theorem}[section]

\newtheorem{lemma}[theorem]{Lemma}

\theoremstyle{definition}
\newtheorem{definition}[theorem]{Definition}

\newcommand{\R}{\mathbb R}
\newcommand{\N}{\mathbb N}

\newcommand{\medint}{-\kern  -,375cm\int}

\newcommand{\h}{\mathcal{H}^{N-1}}

\newcommand{\hs}{d\mathcal{H}^{N-1}}

\title[An isoperimetric result for the fundamental frequency]
      {An isoperimetric result for the fundamental frequency via domain derivative.}

\author[Carlo Nitsch]{ }



 \email{c.nitsch@unina.it}

\begin{document}
\maketitle

\centerline{\scshape \href{http://wpage.unina.it/c.nitsch}{Carlo Nitsch}}
\medskip
{\footnotesize
  \centerline{Dipartimento di Matematica e Applicazioni ``R. Caccioppoli"}
   \centerline{Complesso Monte S. Angelo}
   \centerline{Via Cintia, 80126 Napoli, Italy}
}

\bigskip


\begin{abstract}
The Faber-Krahn deficit $\delta\lambda$ of an open bounded set $\Omega$ is the normalized gap between the values that the first Dirichlet Laplacian eigenvalue achieves on $\Omega$ and on the ball having same measure as $\Omega$. For any given family of open bounded sets of $\R^N$ ($N\ge 2$) smoothly converging to a ball, it is well known that both $\delta\lambda$ and the isoperimetric deficit $\delta P$ are vanishing quantities. It is known as well that, at least for convex sets, the ratio $\frac{\delta P}{\delta \lambda}$ is bounded by below by some positive constant (see \cite{BNT,PW}), and in this note, using the technique of the shape derivative, we provide the explicit optimal lower bound of such a ratio as 
$\delta P$ goes to zero.
\end{abstract}

\section{Introduction}

Given an open bounded set of $\R^N$ its first Dirichlet Laplacian eigenvalue $\lambda$ is the least positive number for which the boundary problem
$$\left\{
\begin{array}{ll}
-\Delta u = \lambda u & \mbox{in }\Omega\\
u=0 & \mbox{on }\partial\Omega
\end{array}
\right.$$
admits nontrivial solutions.

The Faber-Krahn inequality is a remarkable property which, answering to a conjecture formulated by Lord Rayleight, states that among sets of given measure the ball has the least first Dirichlet Laplacian eigenvalue. Namely if $\Omega^\sharp$ denotes the ball having the same measure as $\Omega$ then
\begin{equation}\label{eq_FK}
\lambda(\Omega)\ge\lambda(\Omega^\sharp).
\end{equation}

Inequality \eqref{eq_FK} falls in the large class of so-called isoperimetric inequalities. By antonomasia the isoperimetric inequality is the one which characterizes the ball as the set having minimial perimeter among those sets of fixed volume, but nowadays, in a broad sense isoperimetric inequality is an inequality where a functional is optimized under some geometrical prescription. The study of isoperimetric inequalities goes back to the beginning of mathematics and has always been a flourishing field. Recently many authors turned the attention to the study of quantitative versions of the classical isoperimetric inequalities (see for instance \cite{AFN,CL1,CL2,EFT,FiMP,Fu,FuMP1}), and quantitative versions of Faber-Krahn inequality have been investigated for instance in \cite{FuMP2,Mel}. Here we are interested in a recent result \cite{BNT} obtained in the wake of a celebrated paper by L. E. Payne and H. F. Weinberger \cite{PW}.

As custom let us denote by  
$$\delta P(\Omega)=\frac{Per(\Omega)}{Per(\Omega^\sharp)}-1,$$
the isoperimetric deficit of $\Omega$ and
following \cite{FuMP2} we denote by
$$\delta \lambda(\Omega)=\frac{\lambda(\Omega)}{\lambda(\Omega^\sharp)}-1,$$
the Faber-Krahn deficit.

The classical isoperimetric inequality and the Faber-Krahn inequality respectively infer that both $\delta P$ and $\delta \lambda$ are always non negative quantities.

When $\Omega$ is convex then \eqref{eq_FK} can be improved (see \cite{BNT,PW}) establishing that for any $\eta>0$ there exists $C>0$ depending on $N$ such that if $\delta P(\Omega)\le \eta$ then
\begin{equation}\label{eq_conv}
\delta P(\Omega)\ge C\,\delta  \lambda(\Omega).
\end{equation}

The name ``quantitative Faber-Krahn inequality'' comes from the fact that it quantifies how  ``small" is the Faber-Krahn deficit when the set $\Omega$ is ``close'' to the ball having same measure.

It is easy to show that inequality \eqref{eq_conv} is optimal in the sense that for any $\eta,C,\gamma>0$ there exists a bounded convex set $\Omega$ such that $\delta P(\Omega)\le \eta$ and 
$\delta P(\Omega)< C\,\delta  \lambda(\Omega)^{1+\gamma}.$

Very little is known about the optimal value of the constant $C$ even in the limit as $\delta P\to 0$ and here comes the idea to exploit the technique of shape derivative to investigate the behavior of the ratio $\delta P/\delta \lambda$ along an arbitrary family of sets which converges in a suitable way to a ball. More precisely we use the following definition
\begin{definition}\label{def_family}
We say that a one parameter family $\Omega(t)$ of open bounded sets of $\R^N$ smoothly converges to an open bounded set $\Omega$ as $t$ goes to zero, if there exists a positive $\delta$, and a one parameter family of transformations $\Phi_t$ ($0\le t<\delta$) of $\R^N$ in itself such that
\begin{enumerate}[(a)]
\item $\Phi_t$ and ${\Phi_t}^{-1}$ belong to $C^\infty(\R^N;\R^N)$ for all $0\le t< \delta$;
\item the mappings $t\to \Phi_t (x)$ and $t \to {\Phi_t}^{-1} (x)$ belong to $C^\infty([0,\delta[)$ for all $x\in \R^N$;
\item $\Omega(0)=\Omega$ and $\Omega(t)=\Phi_t(\Omega)$ for all $0\le t<\delta$;
\end{enumerate}
\end{definition}

We denote by $\omega_N$ the volume of the unit ball of $\R^N$, by $J_\nu$ the Bessel function of first kind and order $\nu$, and by $j_{\nu}$ the first positive zero of $J_\nu$. Our main result follows.
\begin{theorem}\label{mainteo}
For any $N\ge 2$ there exists a dimensional constant 
$$C_N=\frac{\displaystyle N(N+1)\int_0^{j_{N/2-1}}rJ^2_{N/2-1}(r)\,dr}{\displaystyle2\left({j_{N/2-1}J_{N/2-1}'}(j_{N/2-1})\right)^2\left(j^2_{N/2-1}-N\right) }$$
such that for any given one parameter family of sets $\Omega(t)$ of $\R^N$, smoothly converging to a ball as $t\to 0$,  then
$$\liminf_{t\to 0} \frac{\delta P(\Omega(t))}{\delta \lambda(\Omega(t))}\ge C_N.$$
The constant $C_N$ is optimal and there exists a family $\Omega(t)$ for which the equality sign is achieved.
\end{theorem}

\section{Proof of Theorem \ref{mainteo}}
We consider a family of open bounded sets $\Omega(t)$ ($0\le t<\delta$) smoothly converging as $t \to 0$ in the sense of Definition \ref{def_family} to a smooth open bounded connected set $\Omega$. For any $0\le t<\delta$ we denote by $\lambda(t)$ and $u(x,t)$ respectively the first Dirichlet Laplacian eigenvalue and the normalized solution to
\begin{equation}\label{sistema_t}
\left\{
\begin{array}{lll}
-\Delta u(x,t)=\lambda(t) u(x,t)  &\mathrm{in}  &\Omega(t)  \\\\
u \ge 0 &\mathrm{in} & \Omega(t)\\\\
u(x,t)=0  & \mathrm{on} & \partial\Omega(t) \\\\
\|u\|_{L^2(\Omega(t))}=1.
\end{array}
\right.
\end{equation}

The proof of Theorem \ref{mainteo} will be carried on by choosing an arbitrary family of smooth sets converging to a ball as $t$ goes to zero and by performing Taylor expansion of the ratio $\frac{\delta P(\Omega(t))}{\delta \lambda(\Omega(t))}$ around $t=0$. 
For the seek of simplicity we split the proof of Theorem \ref{mainteo} in several steps. In the first step we provide the general expression of first and second order derivatives of the first Dirichlet Laplacian eigenvalue along the family $\Omega(t)$. In the second step we let $\Omega$ be the a ball of $\R^N$ and we differentiate $\delta P(\Omega(t))$ and
$\delta \lambda(\Omega(t))$ twice at $t=0$ deducing that there exists a functional $\mathcal{F}$ on $C^\infty(\partial\Omega)$ such that $\dfrac{\delta P(\Omega(t))}{\delta \lambda(\Omega(t))}=\mathcal{F}\left(n\cdot \left.\dfrac{\partial\Phi_t}{\partial t}\right|_{t=0}\right)+o(1)$ as $t$ goes to zero. Here $n$ denotes the unit outer normal to $\partial \Omega$. In the third and last step we show that $C_N$ is exactly the minimum achieved by the functional $\mathcal{F}$ when we vary $\Phi_t$ on the whole class of admissible smooth transformations (in the sense of Definition \ref{def_family} (a)-(b)-(c)).

\emph{Step 1.} We begin the proof computing the first and the second order derivatives of the Dirichlet Laplacian eigenvalue along the family $\Omega(t)$, using the well known Hadamard's formula. Namely we prove the following Lemma
\begin{lemma}\label{lem_seconda}
 For $0\le t<\delta$ let $u(x,t)$ be the family of solutions to \eqref{sistema_t} and let $\lambda(t)$ be the corresponding family of eigenvalues. There exists $\varepsilon>0$ such that for all $0\le t<\varepsilon$ the family $\lambda(t)$ is smooth and it holds
\begin{align}\label{eq_firstvar}
\lambda'(t)=&\int_{\partial\Omega(t)}|Du| \frac{\partial \Phi_t}{\partial t}(\Phi_t^{-1})\cdot Du \,d\h,\\ \label{eq_secondvar}
\lambda''(t)=&\int_{\partial\Omega(t)}|Du| \left(\left[\frac{\partial\Phi_t}{\partial t}(\Phi_t^{-1})\right]^T\cdot D^2 u \cdot\frac{\partial\Phi_t}{\partial t}(\Phi_t^{-1})\right.\\\notag
&\left.+Du \cdot\frac{\partial^2\Phi_{t}}{\partial t^2}(\Phi_t^{-1})+2Dw\cdot\frac{\partial\Phi_t}{\partial t}(\Phi_t^{-1})\right) \,d\h\\ \notag
\end{align}
where $w$ solves
\begin{equation}\label{eq_velocity}
\left\{
\begin{array}{ll}
-\Delta w(x,t)= \lambda'(t) u(x,t)+\lambda(t) w(x,t) & x\in \Omega(t)\\\\
w(x,t)=-\dfrac{\partial \Phi_t}{\partial t}(\Phi_t^{-1}(x))\cdot Du(x,t) & x\in\partial\Omega(t).\\\\
\displaystyle\int_{\Omega(t)}u\, w =0
\end{array}\right.
\end{equation}
\end{lemma}


Equations \eqref{eq_firstvar}-\eqref{eq_secondvar}-\eqref{eq_velocity} are related to other formulas which can be found in literature (see \cite{Hen,HP,NP,SZ} and the references therein contained). In particular \cite{HP,NP} contain very general formulation of the notion of shape derivative with application to the Dirichlet Laplacian eigenvalues. We also observe that very often in literature (see for instance \cite[Theorem 2.5.6]{Hen}) it is considered the case of a first order perturbations of identity, namely $\Phi_t=\mathbb{I}+tW$, where $W$ is a suitably smooth vector field. The result is that the term $Du \cdot\frac{\partial^2\Phi_t}{\partial t^2}$ in \eqref{eq_secondvar}, which in our case plays a crucial rule, would be missing. 

Hence, for the seek of completeness we decided to provide here a complete proof of the statement of Lemma \ref{lem_seconda}, and we exploit a level sets method. 

\begin{proof}[Proof of Lemma \ref{lem_seconda}]
Classical regularity theory \cite{GB} for elliptic equation ensures that, for all $t \in [0,\delta[$, $u(x,t) \in   C^\infty(\overline\Omega(t))$. Moreover, arguing as in \cite[Chapter 5]{HP}, there exists at least some positive $\varepsilon<\delta$ such that  
the function $u(x,t)$ belongs to $C^\infty(   [0,\varepsilon[; C^\infty(\overline\Omega(t)) )$.

\noindent
Differentiating with respect to $t$ the equation in (\ref{sistema_t}) we get
\begin{equation}\label{eq_deltaut}
-\Delta \frac{\partial u}{\partial t}= \lambda'(t) u+\lambda(t) \frac{\partial u}{\partial t}.
\end{equation}
Since $\Omega(t)$ is, at any time, the zero-level set of $u(x,t)$, 
if $y\in\partial \Omega$ then $\Phi_t(y)\in\partial \Omega(t)$ and we have 
\begin{equation}\label{zerolevel}
u(\Phi_t(y),t)=0 \qquad \mbox{for all } y\in\partial\Omega \mbox{ and }  0<t<\varepsilon.
\end{equation}

\noindent
Now, the boundary point $ \Phi_t (y)\in\partial\Omega(t)$ moves with velocity $\dfrac{\partial\Phi_t(y)}{\partial t}$. Differentiating once \eqref{zerolevel} with respect to $t$ we get
$$
\frac{d}{dt}u(\Phi_t(y),t)=Du(\Phi_t(y),t) \cdot\frac{\partial \Phi_t}{\partial t}(y)+\frac{\partial u}{\partial t}(\Phi_t(y),t)=0,
$$
and hence
\begin{equation}\label{eq_boundaryspeed}
Du \cdot\frac{\partial \Phi_t}{\partial t}({\Phi_t}^{-1})+\frac{\partial u}{\partial t}=0\qquad \mbox{on }\partial \Omega(t).
\end{equation}
Therefore the projection of the velocity $\frac{\partial \Phi_t}{\partial t}(y)$ along the direction of the unit outer normal $n$ is equal to $\dfrac{1}{|Du|}\dfrac{\partial u}{\partial t}$, which is pointwise defined since $\Omega(t)$ is smooth and standard barrier arguments imply that $Du$ does not vanish on $\partial \Omega(t)$.

Differentiating twice \eqref{zerolevel} with respect to $t$ we get
\begin{equation}\label{eq_utt}
\left[\frac{\partial \Phi_t}{\partial t}({\Phi_t}^{-1})\right]^T\cdot D^2 u \cdot\frac{\partial \Phi_t}{\partial t}({\Phi_t}^{-1})+Du \cdot\frac{\partial^2 \Phi_t}{\partial t^2}({\Phi_t}^{-1})+2\left(D\frac{\partial u}{\partial t}\right)\cdot\frac{\partial \Phi_t}{\partial t}({\Phi_t}^{-1})+\frac{\partial^2 u}{\partial t^2}=0
\end{equation}
which highlights the connection between the acceleration $\frac{\partial^2 \Phi_t}{\partial t^2}$ of a boundary point $x\equiv\Phi_t(y)\in\partial\Omega(t)$ and the value of $\frac{\partial^2 u}{\partial t^2}$ at the same point.

Now, if $f(x,t)$ is a smooth function and $J(t) = \displaystyle\int_{\Omega(t)} f(x,t) \> dx$, the classical Hadamard formula  gives (see, for instance, \cite{HP, SZ})
\begin{equation}\label{eq_had}
J'(t)= \int_{\Omega(t)} \dfrac{\partial f}{\partial }(x,t) \> dx + \int_{\partial \Omega(t) } f(x,t) \> \dfrac{1}{|Du|}\frac{\partial u}{\partial t} \> \hs.
\end{equation}

\noindent
Therefore, since the $L^2$ norm of $u$ is constant with respect to $t$ and $u$ vanishes on $\partial\Omega(t)$ we have
\begin{align}\label{eq_orto1}
&\frac{d}{dt}\int_{\Omega(t)}u^2 dx=2\int_{\Omega(t)}u \frac{\partial u}{\partial t} dx=0\\\label{eq_orto2}
&\frac{d^2}{dt^2}\int_{\Omega(t)}u^2 dx=2\int_{\Omega(t)}u \frac{\partial^2 u}{\partial t^2} dx+2\int_{\Omega(t)}\left(\frac{\partial u}{\partial t}\right)^2 dx=0.
\end{align}

Furthermore \eqref{eq_had} applied to $\lambda(t)$ provides the relation

$$\lambda'(t)=\frac{d}{d t}\int_{\Omega(t)}|Du |^2\,dx=-\int_{\partial\Omega(t)}|Du| \frac{\partial u}{\partial t} \,d\h=\int_{\partial\Omega(t)}|Du| \frac{\partial \Phi_t}{\partial t}({\Phi_t}^{-1})\cdot Du \,d\h,$$
or equivalently
\begin{equation}\label{eq_first}
\lambda'(t)=\int_{\Omega(t)}Du D\left(\frac{\partial u}{\partial t}\right) \,dx,
\end{equation}
obtaining \eqref{eq_firstvar}.

Finally, if we differentiate $\lambda$ twice, we can use \eqref{eq_deltaut},\eqref{eq_orto1},\eqref{eq_orto2} and \eqref{eq_first} to get
\begin{align}\label{eq_second}
\lambda''(t)&=\frac{d}{dt}\int_{\Omega(t)}Du D\left(\frac{\partial u}{\partial t}\right) \,dx\\\notag
&=\int_{\Omega(t)}\left|D\frac{\partial u}{\partial t}\right|^2 \,dx+\int_{\Omega(t)}Du D\frac{\partial^2 u}{\partial t^2} \,dx+
\int_{\partial\Omega(t)}\frac{Du}{|Du|} \left(D\frac{\partial u}{\partial t}\right) \frac{\partial u}{\partial t}\,d\h\\\notag
&=-\int_{\Omega(t)}\left(\Delta \frac{\partial u}{\partial t}\right) \frac{\partial u}{\partial t} \,dx-\int_{\Omega(t)}\Delta u \frac{\partial^2 u}{\partial t^2} \,dx-
\int_{\partial\Omega(t)}|Du| \frac{\partial^2 u}{\partial t^2} \,d\h\\\notag
&=\lambda(t) \int_{\Omega(t)}\left(\left(\frac{\partial u}{\partial t}\right)^2+ u\frac{\partial^2 u}{\partial t^2}\right) \,dx+\lambda'(t)\int_{\Omega(t)}u \frac{\partial u}{\partial t} \,dx
-\int_{\partial\Omega(t)}|Du| \frac{\partial^2 u}{\partial t^2} \,d\h\\\notag
&=-\int_{\partial\Omega(t)}|Du| \frac{\partial^2 u}{\partial t^2} \,d\h.\\\notag
\end{align}

Once we set $$w(x,t)=\frac{\partial u}{\partial t}(x,t)$$ in \eqref{eq_deltaut}, \eqref{eq_boundaryspeed} and \eqref{eq_orto1}, using \eqref{eq_utt} we get \eqref{eq_secondvar} and the proof is complete
\end{proof}
We observe that the family of transformations $\Phi_t$ is not uniquely determined by the family $\Omega(t)$. In particular it is always possible to choose the velocity vector field $\left.\frac{\partial \Phi_t}{\partial t}\right|_{t=0}$ orthogonal to $\partial\Omega$. 
In such a case \eqref{eq_firstvar}-\eqref{eq_secondvar}-\eqref{eq_velocity} computed at $t=0$ become 
\begin{align}\label{eq_firstvar0}
\lambda'(0)=&-\int_{\partial\Omega}|Du|^2 \,n\cdot\left.\frac{\partial \Phi_t}{\partial t}\right|_{t=0}\\\label{eq_second0}
\lambda''(0)=&\int_{\partial\Omega}\Big[\omega^2H -|Du|^2\,n\cdot \left.\frac{\partial^2 \Phi_t}{\partial t^2}\right|_{t=0}-2\omega\frac{\partial \omega}{\partial n}\Big] \,d\h,
\end{align}
\begin{equation}\label{eq_velocity0}
\left\{
\begin{array}{ll}
-\Delta \omega(x)= \lambda'(0) u(x,0)+\lambda(0) \omega(x) & \mbox{in }\Omega\\\\
\omega(x)=|Du(x,0)|\,n\cdot\left.\dfrac{\partial \Phi_t}{\partial t}(x)\right|_{t=0}  & \mbox{on }\partial\Omega.\\\\
\displaystyle\int_{\Omega}u(x,0)\, \omega(x) \,dx=0
\end{array}\right.
\end{equation}
Here $H$ is the sum of the principal curvatures of $\partial \Omega$ and $n$ the unit outer normal of $\partial\Omega$.
Since $\Delta u(x,0)$ vanishes on $\partial\Omega$ we have used the identity
$$Du^T\cdot D^2u\cdot Du=-\mathrm{div}\left(\frac{Du}{|Du|}\right) |Du|^3 \qquad \mbox{on $\partial\Omega,$} $$
in conjuction with $-\mathrm{div}\left(\frac{Du}{|Du|}\right)=H$. 

\emph{Step 2.} Due to the invariance of both isoperimetric and Faber-Krahn deficits with respect to homotheties we shall perform all the remaining computation under the assumption that the family $\Omega(t)$ has constant volume in $t$ equal to $\omega_N$, therefore from now on $\Omega\equiv\Omega(0)$ is just a unit ball in $\R^n$.

Without loss of generality, we also assume that the velocity field $\left.\dfrac{\partial \Phi_t}{\partial t}\right|_{t=0}$ is orthogonal to $\partial\Omega$ and for all $x\in \partial \Omega$ we denote by $V(x)=n(x)\cdot\left.\dfrac{\partial \Phi_t}{\partial t}\right|_{t=0}$
and by $A(x)=n(x)\cdot \left.\dfrac{\partial^2 \Phi_t}{\partial t^2}\right|_{t=0}$ respectively the initial scalar velocity and the projection of the initial acceleration along the unit outer normal $n(x)$ of $\Omega$. 

Under these assumptions, for $t$ small enough, the boundary of $\Omega(t)$ can be represented in polar coordinates $r\in\R^+$, $\xi\in\mathcal{S}^{N-1}$ by an equation 
$$r(\xi,t)=1+V(\xi)t+\frac{A(\xi)}{2}t^2+o(t^2).$$
If $\sigma_{\xi}$ denotes the usual surface area measure on $\mathcal{S}^{N-1}$ then
$$Per(\Omega(t))=\int_{\mathcal{S}^{N-1}}r(\xi,t)^{N-2}\sqrt{r(\xi,t)^2+|D_{\xi}r(\xi,t)|^2}\, d\sigma_{\xi},$$
and after a taylor expansion we have
\begin{align*}
Per(\Omega(t))=\,&n\omega_N+t\int_{\mathcal{S}^{N-1}}(N-1)V(\xi)\, d\sigma_{\xi}\\
&+\frac{t^2}{2}\int_{\mathcal{S}^{N-1}}\left[(N-1)A(\xi)+(N-1)(N-2)V^2(\xi)+|D_{\xi}V(\xi)|^2\right]\, d\sigma_{\xi}+o(t^2).
\end{align*}
On the other hand, since
$$\omega_N=|\Omega(t)|\qquad \mbox{for all } t\in[0,\delta[$$
then 
\begin{align*}
|\Omega(t)|&=\frac{1}{N}\int_{\mathcal{S}^{N-1}}r(\xi,t)^N\, d\sigma_{\xi}\\
&=\omega_N+t\int_{\mathcal{S}^{N-1}}V(\xi)\, d\sigma_{\xi}+\frac{t^2}{2}\int_{\mathcal{S}^{N-1}}\left[ A(\xi)+(N-1)V^2(\xi)\right]\, d\sigma_{\xi}+o(t^2)
\end{align*}
yields
\begin{equation}\label{eq_constantmeasure}
\int_{\mathcal{S}^{N-1}}V(\xi)\, d\sigma_{\xi}=\int_{\mathcal{S}^{N-1}}\left[ A(\xi)+(N-1)V^2(\xi)\right]\, d\sigma_{\xi}=0.
\end{equation}
As a consequence
$$Per(\Omega(t))=n\omega_N+\frac{t^2}{2}\int_{\mathcal{S}^{N-1}}\left[|D_{\xi}V(\xi)|^2-(N-1)V^2(\xi)\right]\, d\sigma_{\xi}+o(t^2)$$
and
\begin{align*}
\delta P(\Omega(t))&=\frac{Per(\Omega(t))}{N\omega_N^{1/n}|\Omega(t)|^{({N-1})/{N}}}-1\\
&=\frac{t^2}{2N\omega_N}\int_{\mathcal{S}^{N-1}}\left[|D_{\xi}V(\xi)|^2-(N-1)V^2(\xi)\right]\, d\sigma_{\xi}+o(t^2).
\end{align*}

We consider now the series expansion
$$\lambda(t)=\lambda(0)+t\lambda'(0)+\frac{t^2}{2}\lambda''(0)+o(t^2).$$
The gradient $Du(\cdot,0)$ on $\partial\Omega$ has constant modulus (see \cite{Ke})
$$G_N=\frac{j^2_{N/2-1}J'_{N/2-1}(j_{N/2-1})}{\displaystyle\left(N \omega_N \int_0^{j_{N/2-1}}r J^2_{N/2-1}(r)\,dr\right)^{1/2}},$$ and therefore using \eqref{eq_firstvar0} and \eqref{eq_constantmeasure} we deduce $\lambda'(0)=0$ in accordance with the fact that the ball, among sets of fixed measure, is a stationary point for the first Dirichlet Laplacian eigenvalue. 

\noindent Thereafter, for all $x\equiv(r,\xi)\in \Omega$, we set $$v(r,\xi)=G_N\, \omega(x)$$ where $\omega$ is defined in \eqref{eq_velocity0}. 

Taking into account that, for the unit ball, the sum of the principal curvatures $H$ equals $N-1$, 
from \eqref{eq_second0} and \eqref{eq_constantmeasure} we get
$$\lambda''(0)=2G_N^2\int_{\mathcal{S}^{N-1}}\left[V(\xi)\left.\frac{\partial v(r,\xi)}{\partial r}\right|_{r=1}+(N-1)V^2(\xi)\right]\, d\sigma_{\xi}.$$

\noindent Consequently
\begin{align*}
\delta \lambda(\Omega(t))&=\frac{\lambda(t)}{\lambda(0)}-1\\
&=t^2\left(\frac{G_N}{j_{N/2-1}}\right)^2\int_{\mathcal{S}^{N-1}}\left[V(\xi)\left.\frac{\partial v(r,\xi)}{\partial r}\right|_{r=1}+(N-1)V^2(\xi)\right]\, d\sigma_{\xi}+o(t^2).
\end{align*}

Here we have used the fact that the first Dirichlet Laplacian eigenvalue on the unit ball of $\R^N$ is $\lambda(0)=j_{N/2-1}^2$ (see for instance \cite{Ke}) and that $\lambda'(0)=0$.

We need now an explicit representation of the function $v(r,\xi)$ in terms of the scalar velocity $V(\xi)$. To this aim, we observe that \eqref{eq_velocity0}, in conjunction with $\lambda'(0)=0$ and $|Du(x,0)|\Big|_{\partial \Omega}=G_N$ imply that $v(r,\xi)$ satisfies
\begin{equation}\label{sys_Poisson}
\left\{
\begin{array}{ll}
-r^{1-N}\frac{\partial}{\partial r}\left(r^{N-1}\frac{\partial v}{ \partial r}\right)-r^{-2}\Delta_{\xi} v = j_{N/2-1}^2 v & (r,\xi)\in(0,1)\times\mathcal{S}^{N-1}\\\\
v(1,\xi)=G_N\,V(\xi)& \xi\in\mathcal{S}^{N-1}
\end{array}\right.
\end{equation}
where $\Delta_\xi$ is the Laplace Beltrami operator on $\mathcal{S}^{N-1}.$

Then we remind (see for instance \cite{Mu}) that
$V(\xi)$ admits an expansion
$$V(\xi)=\sum_{k=0}^{+\infty}a_kY_k(\xi)\qquad \xi \in \mathcal{S}^{N-1} $$
in terms of a family of spherical harmonics $\{Y_k(\xi)\}_{k\in\N}$ which satisfy for all $k\ge 0$ 
$$-\Delta_\xi Y_k = k(k+N-2)Y_k  \quad\mbox{and}\quad \|Y_k\|_{L^2}=1.$$
The coefficient $a_k$ is the projection of $V$ on the normalized eigenfuntion $Y_k$
$$a_k=\int_{\mathcal{S}^{N-1}}V(\xi) Y_k(\xi)\, d\sigma_\xi,$$
so that
$$\|V\|^2_{L^2}=\sum_{k=0}^{+\infty}a^2_k.$$
Notice that $Y_0=(N\omega_N)^{-1/2}$ and \eqref{eq_constantmeasure} imply 
$$a_0=\int_{\mathcal{S}^{N-1}}V\,Y_0\, d\sigma_{\xi}=(N\omega_N)^{-1/2}\int_{\mathcal{S}^{N-1}}V\, d\sigma_{\xi}=0.$$

Accordingly we use the separation of variables $v(r,\xi)=\sum_kR_k(r)Y_k(\xi)$ to solve the Poisson problem \eqref{sys_Poisson} and infer
$$v(r,\xi)=r^{1-N/2}\sum_{k\ge1}a_k\frac{J_{\ell_k} (j_{N/2-1}\,r)}{J_{\ell_k} (j_{N/2-1})}Y_k(\xi)$$
where
$\ell_k=\sqrt{k(k+N-2)+(N/2-1)^2}=k+N/2-1.$

Consequently we have
$$\frac{\partial v}{\partial r}(1,\xi)=\sum_{k\ge1}\left[\left(1-\frac{N}{2}\right)+j_{N/2-1}\frac{J'_{\ell_k} (j_{N/2-1})}{J_{\ell_k} (j_{N/2-1})}\right]a_kY_k(\xi),$$
and in view of the recurrence relations of the Bessel functions (see \cite[\S 9.1.27]{AS})
$$J'_\ell(s)=\frac{\ell}{s}J_{\ell}(s)-J_{\ell+1}(s)$$
we can write
$$\frac{\partial v}{\partial r}(1,\xi)=\sum_{k\ge1}\left[k-j_{N/2-1}\frac{J_{\ell_k+1} (j_{N/2-1})}{J_{\ell_k} (j_{N/2-1})}\right]a_kY_k(\xi).$$
Finally we obtain

$$\frac{\delta P(\Omega(t))}{\delta \lambda(\Omega(t))}=\left(\frac{j_{N/2-1}^2 }{2N\omega_N G_N^2}\right)\frac {\displaystyle \sum_{k\ge2}a_k^2 \left[\ell_k^2-\frac{N^2}{4}\right] }{ \displaystyle\sum_{k\ge2}a_k^2 \left[\ell_k+\frac{N}{2}-j_{N/2-1}\frac{J_{\ell_k+1} (j_{N/2-1})}{J_{\ell_k} (j_{N/2-1})}\right]  }+o(1).$$
Observe that $a_1$ provides no contribution in the summation, indeed the projection of $V$ on the subspace $Y_1$ corresponds to a translation of the ball $\Omega(0)$ (with no deformation).

\emph{Step 3.} It is evident that 
\begin{equation}\label{eq_liminf}
\liminf_{t\to 0}\frac{\delta P(\Omega(t))}{\delta \lambda(\Omega(t))}\ge \inf_{k\ge2}  \left[\left(\frac{j_{N/2-1}^2 }{2N\omega_N G_N^2}\right)\frac {\ell_k^2-\frac{N^2}{4} }{\ell_k+\frac{N}{2}-j_{N/2-1}\displaystyle \frac{J_{\ell_k+1} (j_{N/2-1})}{J_{\ell_k} (j_{N/2-1})} }\right]
\end{equation}
and the remainder of the proof of Theorem \ref{mainteo} consists in showing that the infimum on righthand side of \eqref{eq_liminf} is achieved for $k=2$ independently on $N$.
In fact the constant $C_N$ defined in Theorem \ref{mainteo} coincides with
$$\left.\left[\left(\frac{j_{N/2-1}^2 }{2N\omega_N G_N^2}\right)\frac {\ell_k^2-\frac{N^2}{4} }{\ell_k+\frac{N}{2}-j_{N/2-1}\displaystyle \frac{J_{\ell_k+1} (j_{N/2-1})}{J_{\ell_k} (j_{N/2-1})} }\right]\right|_{k=2}.$$

In principle, minimizing the righthand side of \eqref{eq_liminf} is elementary. However it is worth providing the details, since the proof involves the usage of several nontrivial properties of the Bessel functions.

The next Lemma concludes the proof of Theorem \ref{mainteo}. In what follows we use the notion of convex sequence:
\begin{definition}
We say that a sequence $\{\alpha_k\}_{k\in\N}$ of real numbers is convex (concave) for $k\ge k_0$ if $\alpha_{k+1}-2\alpha_k+\alpha_{k-1}\ge 0$ ($\le 0$) when $k\ge k_0+1$.
\end{definition}
\begin{lemma}\label{lemma_minimal}
For all $N\ge 2$, let $k\ge 2$, 
$\ell_k=k+N/2-1$, and
$$\mathcal{Q}_k=\left(\frac{j_{N/2-1}^2 }{2N\omega_N G_N^2}\right)\frac {\ell_k^2-\frac{N^2}{4} }{\ell_k+\frac{N}{2}-j_{N/2-1}\displaystyle \frac{J_{\ell_k+1} (j_{N/2-1})}{J_{\ell_k} (j_{N/2-1})} }.$$
We have $\mathcal{Q}_k\le\mathcal{Q}_{k+1}.$

\end{lemma}

\begin{proof}[Proof of Lemma \ref{lemma_minimal}]
First we prove that, for any given value $N\ge 2$, the sequence $\left\{\frac{J_{\ell_k+1} (j_{N/2-1})}{J_{\ell_k}(j_{N/2-1})}\right\}_{k\in \N}$ is positive, decreasing, vanishing, and convex for $k\ge 2$.
Denoting by $z_N=j_{N/2-1}$, the claim follows at once from the continued fraction representation (see \cite[\S 9.1.73]{AS}) 
\begin{equation}\label{f_continua}
  \frac{J_{\ell_k+1}(z_N)}{J_{\ell_k}(z_N)} =  \cfrac{1}{\dfrac{2(\ell_k+1)}{z_N}
          - \cfrac{1}{\dfrac{2(\ell_k+2)}{z_N}
          - \cfrac{1}{\dfrac{2(\ell_k+3)}{z_N} - \ddots}}}
\end{equation}
In fact, after observing that $\dfrac{2(\ell_k+1)}{z_N}>1$ (see \cite{Cha,QW}) it is easy to deduce that the sequence $\left\{\frac{J_{\ell_k+1} (j_{N/2-1})}{J_{\ell_k}(j_{N/2-1})}\right\}_{k\in \N}$ is positive decreasing and vanishing. It remains to prove the convexity and we begin observing that if a sequence $\{\alpha_k\}_{k\in\N}$ is concave then $\{\alpha_k^{-1}\}_{k\in\N}$ is convex, hence
\begin{equation*}
 \cfrac{1}{\dfrac{2(\ell_k+1)}{z_N}}, \qquad k\in\N
 \end{equation*}
 is convex for $k\ge 2$ but also
 \begin{equation*}
 \cfrac{1}{\dfrac{2(\ell_k+1)}{z_N}
          - \cfrac{1}{\dfrac{2(\ell_k+2)}{z_N}}},\qquad\mbox{and}\qquad
           \cfrac{1}{\dfrac{2(\ell_k+1)}{z_N}
          - \cfrac{1}{\dfrac{2(\ell_k+2)}{z_N}
         - \cfrac{1}{\dfrac{2(\ell_k+3)}{z_N} }}}\qquad k\in\N
\end{equation*}
are convex for $k\ge 2$, as well as any truncation of the continued fraction \eqref{f_continua}. By approximation we deduce that $\left\{\frac{J_{\ell_k+1} (j_{N/2-1})}{J_{\ell_k}(j_{N/2-1})}\right\}_{k\in \N}$ is a convex sequence for $k\ge 2$. 

\noindent Eventually we deduce that, for any $N$ fixed, $\mathcal{Q}_k$ is a convex sequence for $k\ge 2$ since it is the ratio between a positive convex sequence 
$\left\{\left(\frac{z_N^2 }{2N\omega_N G_N^2}\right)\left({\ell_k^2-\frac{N^2}{4}}\right)\right\}_{k\in\N}$ and a positive concave sequence $\left\{\ell_k+\frac{N}{2}-z_N\displaystyle \frac{J_{\ell_k+1} (z_N)}{J_{\ell_k} (z_N)}\right\}_{k\in\N}$.

The convexity of $\mathcal{Q}_k$ for $k\ge 2$ implies the increasing monotonicity of $\mathcal{Q}_k$ for $k\ge 2$ if and only if $\mathcal{Q}_2\le\mathcal{Q}_3$. 
In view of the recurrence relations of the Bessel functions \cite[\S 9.1.27]{AS}, we have $$\frac{J_{\ell_2}(z_N)}{J_{\ell_1}(z_N)}=\frac{N}{z_N},$$ $$\frac{J_{\ell_3}(z_N)}{J_{\ell_2}(z_N)}=\frac{N+2}{z_N}-\frac{z_N}{N},$$
and $$\frac{J_{\ell_4}(z_N)}{J_{\ell_3}(z_N)}=\frac{N+4}{z_N}-\left(\frac{N+2}{z_N}-\frac{z_N}{N}\right)^{-1}.$$ After a tedious but straightforward computation we get $\mathcal{Q}_2<\mathcal{Q}_3$ if and only if
\begin{equation}\label{eq_finale}
(z_N^2-2)N^2+5z_N^2N-2z_N^4>0.
\end{equation}
It is not difficult to prove that the last inequality holds true for large values of $N$, in view of the following upper and lower bounds on the first zero of Bessel functions (see \cite{Cha,QW})
\begin{equation}\label{eq_upperbound}
\frac{N}{2}-1\le z_N \le \sqrt{\frac N2}\left(\sqrt{\frac N2+1}+1\right).
\end{equation}
In fact plugging \eqref{eq_upperbound} in \eqref{eq_finale} we have
\begin{align*}
&(z_N^2-2)N^2+5z_N^2N-2z_N^4\\
&= z_N^2(N^2-2z_N^2+2N)+N(3z_N^2-2N)\\
&\ge z_N^2(N^2-N(\sqrt{N/2+1}+1)^2+2N)+N(3(N/2-1)^2-2N)\\
&=\frac{Nz_N^2}{2}\left(N-2^{3/2}\sqrt{N+2}\right)+\frac{3}{4}N(N-6)\left(N-\frac{2}{3}\right)
\end{align*}
and it is not difficult to prove that the last quantity is increasing for $N\ge10$ and positive for $N=10$ therefore positive for $N\ge 10$.

\noindent For the remaining values $2\le N\le 9$, it is elementary to check inequality \eqref{eq_finale} using the following table  
\begin{center}
\begin{tabular}{|c|c|}
\hline
$N$ & $z_N\equiv j_{N/2-1}$\\
\hline
2 &  $\approx$ 2.404826\\
3 & $\pi$\\
4 & $\approx$ 3.831706\\
5 & $\approx$ 4.4934095\\
6 & $\approx$ 5.135622\\
7 & $\approx$ 5.763459\\
8 & $\approx$ 6.380162\\
9 & $\approx$ 6.987932\\
\hline
\end{tabular}
\end{center}
and the proof is complete.
\end{proof} 
According to the compatibility condition \eqref{eq_constantmeasure}, by construction it is clear that $$\liminf_{t\to 0}\frac{\delta P(\Omega(t))}{\delta \lambda(\Omega(t))}=C_N$$ provided $$V(\xi)=Y_2(\xi) \qquad \mbox{and}\qquad \int_{\mathcal S^{N-1}}A(\xi)\,d\sigma_\xi=(1-N).$$

\end{document}